\documentclass{amsart}
\usepackage{mathrsfs}

\usepackage{enumitem}

\theoremstyle{plain}
\newtheorem{theorem}{Theorem}[section] 
\newtheorem{thm}{Theorem}[section] 
\newtheorem{conj}[thm]{Conjecture} 
\newtheorem{lem}[thm]{Lemma}  
\newtheorem{cor}[thm]{Corollary}

\theoremstyle{definition}

\newtheorem{example}[theorem]{Example}

\theoremstyle{remark}
\newtheorem{rem}[theorem]{Remark}

\numberwithin{equation}{section}

\newcommand{\ps}{\pi^{\xi}}
\newcommand{\e}{\mathbf{e}}
\newcommand{\s}{\mathbf{s}}

\DeclareMathOperator{\amaj}{amaj}
\DeclareMathOperator{\asc}{asc}
\DeclareMathOperator{\A}{Asc}
\DeclareMathOperator{\comaj}{comaj}
\DeclareMathOperator{\cyc}{cyc}
\DeclareMathOperator{\des}{des}
\DeclareMathOperator{\D}{Des}
\DeclareMathOperator{\desB}{\des_B}
\DeclareMathOperator{\exc}{exc}
\DeclareMathOperator{\finv}{finv}
\DeclareMathOperator{\fmaj}{fmaj}
\DeclareMathOperator{\Ifmaj}{afmaj}
\DeclareMathOperator{\inv}{inv}
\DeclareMathOperator{\maj}{maj}
\DeclareMathOperator{\stat}{stat}

\newcommand{\integers}{\mathbb{Z}} 
\newcommand{\reals}{\mathbb{R}} 
\newcommand{\Sn}{{\mathfrak S}_n} 
\newcommand{\Bn}{{\mathfrak B}_n} 
\newcommand{\Dn}{{\mathfrak D}_n} 
\newcommand{\w}{\integers_k \wr {\mathfrak S}_n} 

\newcommand{\PP}{\mathscr{P}} 
\newcommand{\I}{\mathfrak{I}} 
\newcommand{\E}{{E}} 
\newcommand{\poly}{E} 
\newcommand{\h}{\mathbf{h}} 

\newcommand{\la}{\lambda}
\newcommand{\floor}[1]{\ensuremath{\left\lfloor #1 \right\rfloor}} 

\begin{document}

\title{The $\s$-Eulerian polynomials have only real roots}
\author{Carla D. Savage}
\address{Department of Computer Science, North Carolina State University, Raleigh, North Carolina 27695-8206}
\curraddr{}
\email{savage@ncsu.edu}
\thanks{}
\author{Mirk\'o Visontai}
\address{Department of Mathematics, University of Pennsylvania,
Philadelphia, Pennsylvania 19104}
\curraddr{Department of Mathematics, Royal Institute of Technology,
SE-100 44 Stockholm, Sweden}
\email{visontai@kth.se}
\thanks{}

\date{}

\dedicatory{}


\begin{abstract}
We study the roots of generalized Eulerian polynomials via a novel approach. We interpret Eulerian polynomials 
as the generating polynomials of a statistic over inversion sequences. Inversion sequences (also known as Lehmer 
codes or subexcedant functions) were recently generalized by Savage and Schuster, to arbitrary sequences $\s$ of 
positive integers, which they called $\s$-inversion sequences.

Our object of study is the generating polynomial of the {\em ascent} statistic over the set of $\s$-inversion 
sequences of length $n$. Since this ascent statistic over inversion sequences is equidistributed with the descent 
statistic over permutations we call this generalized polynomial the \emph{$\s$-Eulerian polynomial}. The main 
result of this paper is that, for any sequence $\s$ of positive integers, the $\s$-Eulerian polynomial has only real roots.

This result is first shown to generalize several existing results about the real-rootedness of various Eulerian 
polynomials. We then show that it can be used to settle a conjecture of Brenti, that Eulerian polynomials for 
all finite Coxeter groups have only real roots, and partially settle a conjecture of Dilks, Petersen, Stembridge 
on type B affine Eulerian polynomials. 
It is then extended to several $q$-analogs. We show that 
the MacMahon--Carlitz $q$-Eulerian polynomial has only real roots whenever $q$ is a positive real number confirming 
a conjecture of Chow and Gessel. The same holds true 
for the hyperoctahedral group and the wreath product groups, confirming 
further conjectures of Chow and Gessel, and Chow and Mansour, respectively. 

Our results have interesting geometric consequences as well.

\end{abstract}

\maketitle

\section{Introduction}

For a sequence $\s = \{s_i\}_{i\ge 1}$  of positive integers,
the $n$-dimensional \emph{$\s$-inversion sequences}   are defined by
\[\I_n^{(\s)} = \{(e_1, \dotsc, e_n) \in \mathbb{Z}^n \mid  
 0\le e_i < s_i \;\;\text{for}\; 1\le i \le n\}.\]
 The \emph{ascent set} of an $\s$-inversion sequence 
$\e = (e_1, \dotsc, e_n) \in \I_n^{(\s)}$ is  the set
\begin{equation}
\A \e = \left\{i \in \{0,1, \ldots, n-1\} \Bigm| \frac{e_{i}}{s_{i}} < \frac{e_{i+1}}{s_{i+1}} \right\},
\label{Ascdef}
\end{equation}
with the convention that $e_0=0$ (and $s_0 = 1$). 
The  \emph{ascent statistic} on $\e \in \I_n^{(\s)}$ is
\[
\asc \e = \left|\A\e\,\right|.
\]

When $\s=(1,2,3, \ldots )$, there are well-known bijections between $\I_n^{(\s)}$, the
set of inversion sequences and
$\Sn$, the set of permutations of $\{1,2, \ldots, n\}$.
We use $\I_n^{(\s)}$ to generalize results about the distribution of statistics on $\Sn$.
The  generating polynomial of the descent statistic is the {\em Eulerian polynomial} defined as
\[A_n(x) \  := \  \sum_{\pi \in \Sn} x^{\des \pi}\,. \]
 Here, $\des \pi$ is the number of indices 
$i \in \{1,2, \ldots, n-1\}$ such that $\pi_i > \pi_{i+1}$.

Our object of study is the generating polynomial of the ascent statistic over the set of $\s$-inversion sequences of length $\I_n^{(\s)}$: 
\[ \E_n^{(\s)}(x) = \sum_{\e \in \I_{n}^{(\s)}} x^{\asc \e}. \]
Since this ascent statistic over inversion sequences for $\s=(1,2,3, \ldots )$ is equidistributed with the descent statistic over permutations (see \cite[Lemma 1]{SS}) we call this generalized polynomial the \emph{$\s$-Eulerian polynomial}. 

In addition to its many remarkable properties, $A_n(x)$ is known to have {\em only real roots} \cite{Fro},
a property which implies that its coefficient sequence is unimodal and log-concave.

\medskip
Our main result is the following generalization.

\begin{thm}
Let $\s$ be any sequence of positive integers and $n$ a positive integer.
Then the $\s$-Eulerian polynomial
\[
\E_n^{(\s)}(x) = \sum_{\e \in \I_{n}^{(\s)}}  x^{\asc \e} \]
has only real roots.
\label{thm:main}
\end{thm}

In Section~\ref{sec:main}, we prove Theorem~\ref{thm:main} by refining the $\s$-Eulerian polynomials
in a way that allows for an easy recurrence. Using a result of Chudnovsky and Seymour \cite{ChudnovskySeymour}, we 
show that the recurrence preserves a stronger property---that these refined polynomials are \emph{compatible}---which in turn implies the theorem. 

Variations of the Eulerian polynomials  arise as descent generating polynomials in combinatorial families other 
than permutations. As we show in Section~\ref{sec:applications}, Theorem~\ref{thm:main} 
generalizes many previous results concerning the real-rootedness of these Eulerian polynomials.
It also implies some new results.  Notably, by extending our main theorem we are able to settle 
a conjecture of Brenti on the Eulerian polynomials of Coxeter groups \cite{brenti94} 
and partially settle a related 
conjecture of Dilks, Petersen, and Stembridge for affine Eulerian polynomials of Weyl groups \cite{DPS09}.

In Section~\ref{sec:geometry}, we discuss the geometric significance of
Theorem~\ref{thm:main}. The above mentioned Eulerian polynomials are known to be the $h$-polynomials
of Coxeter complexes and the affine Eulerian polynomials are the $h$-polynomials of the reduced Steinberg tori. A different geometric connection can be obtained
by considering the {\em $\s$-lecture hall polytope}, $\PP_n^{(\s)}$ which is defined by
\[
\PP_n^{(\s)} = \left\{ (\la_1, \la_2, \dotsc, \la_n) \in \reals^n \ \Big| \
0 \leq \frac{\la_{1}}{s_{1}}
\leq \frac{\la_{2}}{s_{2}} \leq \cdots
\leq \frac{\la_{n}}{s_{n}} \leq 1 \right\}.
\]
It is known (see \cite{SS}) that the $\s$-Eulerian polynomial is the $\h^*$-polynomial of $\PP_n^{(\s)}$.

In Section~\ref{sec:q-analogs},
we extend Theorem~\ref{thm:main}
 to a $(p,q)$-analog of $\E_n^{(\s)}(x)$.
With the help of this extension,
we show, for the first time, that the  MacMahon--Carlitz $q$-Eulerian polynomial has only real roots
for positive  real $q$, a result conjectured by Chow and Gessel in \cite{ChowGessel}.
We further show that
several other $q$-Eulerian polynomials for signed permutations, and the wreath products $\w$
(colored permutations, indexed permutations) are real-rooted for positive $q$.
This includes the generating polynomial for the joint distribution of descent and flag-inversion number.
We also study the generating polynomial for the joint distribution of descent and flag-major index on signed permutations and the wreath products. We prove that this $q$-analog also has all roots real for positive $q$, a result which was conjectured by Chow and Gessel \cite{ChowGessel}  for signed permutations and by Chow and Mansour \cite{ChowMansour}, for $\w$.

\section{The main result}
\label{sec:main}
In this section, we prove Theorem~\ref{thm:main} using the method of ``compatible polynomials'' in conjunction with a recurrence for $\E^{(\s)}_n(x)$. We also discuss some
connections of our results to previous work and the more familiar notion of interlacing (of roots).

\subsection{A recurrence for the $\s$-Eulerian polynomial}

Let $\chi(\varphi)$ be $1$ if the statement $\varphi$ is true and $0$ otherwise.
In order to show that the $\s$-Eulerian polynomial
has all real roots, consider a refinement:
\begin{equation}
\poly^{(\s)}_{n,i}(x) := \sum_{\e \in \I_{n}^{(\s)}}
\chi(e_n = i)\, x^{\asc \e}\,.
\label{eq:Pni}
\end{equation}
Clearly,  
\begin{equation}
\E^{(\s)}_{n}(x) = \sum_{i=0}^{s_n-1} \poly^{(\s)}_{n,i}(x).
\label{Esum}
\end{equation}
The benefit of introducing these polynomials is that they satisfy a simple recurrence, which we now prove.
\begin{lem}
Given a sequence $\s = \{s_i\}_{i \ge 1}$ of positive integers, let
$n \geq 1$ and $0 \leq i < s_n$.
Then, for all $n > 1$, we have the recurrence
\begin{equation}
\poly^{(\s)}_{n,i}(x) \  = \   \sum_{h=0}^{t_i-1} x \poly^{(\s)}_{n-1,h}(x)  \ +
\sum_{h=t_i}^{s_{n-1}-1}  \poly^{(\s)}_{n-1,h}(x),
\label{eq:recurrencePni}
\end{equation}
where $t_i = \lceil i s_{n-1}/s_{n} \rceil,$ 
with initial conditions 
$\poly^{(\s)}_{1,0}(x)=1$ and $\poly^{(\s)}_{1,i}(x)=x$ for $0 < i < s_1$.

\end{lem}
\begin{proof}
Consider an inversion sequence $\e=(e_1, \dotsc, e_{n-1}, e_{n}) \in  \I_{n}^{(\s)}$
with  $e_n=i$.
By definition (\ref{Ascdef}) of the ascent set, 
$n-1 \in \A \e$ if and only if
$e_{n-1}/s_{n-1} < i/s_n$,
or, equivalently, if and only if
$0 \leq e_{n-1} \leq  \lceil i s_{n-1}/s_{n} \rceil - 1$ holds. So,
\[ \asc (e_1, \ldots,e_{n-1},i)  =  \asc (e_1, \ldots, e_{n-1})+ \chi(e_{n-1} \leq t_i -1),
\]
which proves (\ref{eq:recurrencePni}) by setting $h= e_{n-1}$. For the initial conditions, recall that $e_0/s_0 = 0$, by definition, and hence
$0 \in \A \e$ if and only if  $e_1 > 0$.

\end{proof}

\begin{rem} For the special case of the classical Eulerian polynomials 
$A_n(x) = \poly^{(1,2, \dots, n)}_n(x)$, essentially 
the same refinement $A_{n,i}(x)$ was considered  in a geometric context 
in \cite[Section 4]{NPT11} under the name \emph{restricted
Eulerian polynomials}. Thanks to Eran Nevo for bringing this to our attention.
\end{rem}

\subsection{Compatible polynomials}

Polynomials $f_1(x), \dotsc, f_m(x)$ over $\reals$ are \emph{compatible} if, 
for all real $c_1, \dotsc, c_m \ge 0$,
the polynomial \[\sum_{i=1}^m c_if_i(x)\] has only real roots.
We call such a weighted sum $\sum_{i=1}^m c_if_i(x)$ of polynomials, with nonnegative coefficients 
$c_1, \dotsc, c_m$
a {\em conic combination} of $f_1(x), \dotsc, f_m(x)$.
A real-rooted polynomial (over $\reals$) is compatible with itself.
The polynomials $f_1(x), \dotsc, f_m(x)$ are \emph{pairwise compatible}
if for all $i,j \in \{1,2, \dots, m\}$, $f_i(x)$ and $f_j(x)$ are compatible.

The following lemma is useful in proving that a collection of polynomials is compatible.
\begin{lem}[Chudnovsky--Seymour \cite{ChudnovskySeymour}, 2.2]
\label{lem:ChudnovskySeymour}
The polynomials $f_1, \dotsc, f_m$ with positive leading coefficients are pairwise compatible if and only if they are compatible.
\end{lem}

\subsection{Proof of Theorem~\ref{thm:main}}

We will prove Theorem~\ref{thm:main} by establishing the following---more general---theorem. Theorem~\ref{thm:main} then follows in view of
(\ref{Esum}), (\ref{eq:recurrencePni}) and Lemma \ref{lem:ChudnovskySeymour}.
\begin{thm}
\label{thm:compatible}
Given a set of polynomials $f_1, \dotsc, f_m \in \mathbb{R}[x]$ with positive leading
coefficients satisfying
for all $1\le i < j \le m$ that
\begin{enumerate}[label=\emph{(\alph*)}, ref=(\alph*)]
\item $f_i(x)$ and $f_j(x)$ are compatible, and \label{f1}
\item $xf_i(x)$ and $f_j(x)$ are compatible \label{f2}
\end{enumerate}
define another set of polynomials $g_1, \dotsc, g_{m'}\in \mathbb{R}[x]$ by the equations
\[g_k(x) = \sum_{\ell=0}^{t_k-1} xf_\ell(x) + \sum_{\ell=t_k}^{m}f_\ell(x),\quad \mathrm{for}\; 1\le k \le m'\]
where $0\le t_0\le t_1 \le \dotso \le t_{m'} \le m$. Then, for all $1\le i < j \le m'$ 
\begin{enumerate}[label=\emph{(\alph*')}, ref=(\alph*')]
\item $g_i(x)$ and $g_j(x)$ are compatible, and \label{g1}
\item $xg_i(x)$ and $g_j(x)$ are compatible. \label{g2}
\end{enumerate}
\end{thm}
\begin{proof} 
We first show \ref{g1},
i.e., that the polynomial
$c_ig_i(x) + c_jg_j(x)$ has only real roots for all 
$c_i, c_j \ge 0.$ By the definition of $g_i(x)$, $g_j(x)$ and the assumption that $t_i \le t_j$
it is clear that 
\[c_ig_i(x) + c_jg_j(x) = \sum_{\alpha=0}^{t_i-1}(c_i + c_j)xf_\alpha(x) + 
\sum_{\beta = t_i}^{t_j-1} (c_i+c_jx)f_\beta(x) + 
\sum_{\gamma=t_j}^{m}(c_i + c_j)f_\gamma(x),\] 
that is, $c_ig_i(x) + c_jg_j(x)$ can be written as a conic 
combination of the following polynomials, which we group into three (possibly empty) sets:
\[\left\{xf_\alpha(x)\right\}_{0\le \alpha < t_i} \  \cup \  
\left\{(c_i+c_jx)f_\beta(x)\right\}_{t_i \le \beta < t_j } \ \cup  \ 
\left\{f_\gamma(x)\right\}_{t_j \le \gamma \le m}\,.\]
Therefore, it suffices to show that these $m$
polynomials are compatible.
In fact, by Lemma~\ref{lem:ChudnovskySeymour},
it is equivalent to show that they
are pairwise compatible. This is what we do next.

First, two polynomials from the same sets are compatible by \ref{f1}. 
Secondly, a polynomial from the first set is compatible with another from the third set by 
\ref{f2}, since $\alpha < \gamma$. To show compatibility between a polynomial from the first set and one from the second, we need that 
$a x f_\alpha(x) + b (c_i  + c_jx) f_\beta(x)$ has only real roots 
for all $a,b,c_i,c_j \ge 0$ and $\alpha < \beta$.
This expression is a conic combination of 
$xf_\alpha(x)$, $xf_\beta(x)$, and $f_\beta(x)$.
Since $\alpha < \beta$, 
these three polynomials 
are again pairwise 
compatible by \ref{f1} and \ref{f2} (and the basic fact the $f(x)$ and $xf(x)$ are compatible),
and hence compatible, by Lemma~\ref{lem:ChudnovskySeymour}. 
Finally, the compatibility of a polynomial in the second set and one in the third set follows by a similar argument, 
exploiting the fact that, $xf_\beta(x)$, $f_\beta(x)$, and $f_\gamma(x)$ are pairwise 
compatible for $\beta < \gamma$.

Now we are left to show \ref{g2}, that $xg_i(x)$ and $g_j(x)$ are compatible for all $i < j$. 
This is done in a similar manner. In order to show that $c_ixg_i(x)+ c_jg_j(x)$ 
is real-rooted for all $c_i, c_j \ge 0$ we show that
\[
\left\{x(c_ix+c_j)f_\alpha(x)\right\}_{0\le \alpha < t_i} \  \cup  \  
\left\{xf_\beta(x)\right\}_{t_i \le \beta < t_j} \ \cup \   
\left\{(c_ix+c_j)f_\gamma(x)\right\}_{t_j \le \gamma \le m}
\]
is a set of compatible polynomials, which follows from analogous reasoning to the above. 
Two polynomials from the same subsets are compatible by \ref{f1}.
 Considering one from the first and one 
from the third subset:  $xf_\alpha(x)$ 
and  $f_\gamma(x)$ are compatible by \ref{f2}, since $\alpha < \gamma$. 
Similarly, $x^2f_\alpha(x)$, $xf_\alpha(x)$, and 
$xf_\beta(x)$  
are pairwise compatible which settles the case when we have a polynomial from the first and one 
from the second subset. Finally, $xf_\beta(x)$, $xf_\gamma(x)$, and $f_\gamma(x)$ are compatible, settling the case of
one polynomial from the second subset and one from the third.

\end{proof}

\begin{proof}[Proof of Theorem~\ref{thm:main}]
We use induction on $n$.
When $n=1$, for $0 \leq i \leq j < s_1$,
\[(E^{(\s)}_{1,i}(x),E^{(\s)}_{1,j}(x)) \in
\{(1,1), \ (1,x), \ (x,x)\}\]
and thus
\[(xE^{(\s)}_{1,i}(x), E^{(\s)}_{1,j}(x)) \in
\{(x,1), \ (x,x), \ (x^2,x)\}\,.\]
Clearly, each of the pairs of polynomials
$(1,1)$, 
$(1,x)$, 
$(x,x)$, 
$(x^2,x)$, is compatible. 
From (\ref{eq:recurrencePni}) we see that the polynomials 
$E^{(\s)}_{n,i}(x)$ satisfy a recurrence of the form required in 
Theorem~\ref{thm:compatible}. Hence, by induction, they are compatible
for all $n$ and $0\le i < s_n$. 
In particular, $E^{(\s)}_{n}(x)$ has only real roots for $n\ge 1$.

\end{proof}

\subsection{Connection to interlacing}
We now make a small detour to discuss the connection of compatibility to interlacing, 
and mention some related work in this direction.

Given $f(x) = \prod_{i=1}^{\deg f} (x -x_i)$ and 
$g(x) = \prod_{j=1}^{\deg g} (x- \xi_j)$, two real-rooted polynomials, 
we say that \emph{$f$ interlaces $g$} if their roots alternate in the following way
 \begin{equation}
 \dots \le x_2 \le \xi_2 \le x_1 \le \xi_1\,.
 \label{eq:roots-interlacing}
 \end{equation}
Note that this requires the degrees of $f$ and $g$ to satisfy the following inequalities:
 $\deg f \le \deg g \le \deg f + 1.$ In particular, the order of polynomials is important.

Interlacing of two polynomials implies the real-rootedness of their arbitrary linear combination
by the famous theorem of Obreschkoff.
\begin{thm}[Satz 5.2 in \cite{obreschkoff}]
Let $f,g \in \mathbb{R}[x]$ with $\deg f \le \deg g \le \deg f + 1$. Then $f$ interlaces 
$g$ if and only if their arbitrary linear combination, 
$c_1f(x) + c_2g(x)$ for all $c_1,c_2 \in 
\mathbb{R}$ has only real roots.
\end{thm}

In Theorem~\ref{thm:compatible}, if we require that the $f_1, \dotsc, f_m$ polynomials have
only nonnegative coefficients we can simplify the conditions \ref{f1} and \ref{f2} using the notion of
interlacing. 
(Note that this will not be much of a restriction for us as all polynomials considered in this paper have 
this property.) 
The following lemma is due to D.~G.~Wagner. This version appeared (without a proof) in 
\cite[Lemma~3.4]{Wag00} where it was also mentioned that
it can be proved with the same techniques that were used to obtain \cite[Corollary~5.3]{Wag92}, the special case of the lemma 
when $\deg f = \deg g.$
\begin{lem}
Let $f,g \in \mathbb{R}[x]$ be polynomials with nonnegative coefficients. Then the
 following two statements are equivalent:
\begin{enumerate}[label=\emph{(\roman*)}]
\item $f(x)$ and $g(x)$ are compatible, and $xf(x)$ and $g(x)$ are also compatible.
\item $f(x)$ interlaces $g(x)$.
\end{enumerate}
\label{lem:interlacing}
\end{lem}
\begin{proof}
Let $n_f(x_0)$ denote the number of roots of the 
polynomial $f$ in the interval $[x_0,\infty)$. There exists an equivalent formulation for
both compatibility and interlacing in terms of this notion. First, $f$ and $g$
are compatible if and only if $\left|n_f(x_0) - n_g(x_0)\right| \le 1$ for all $x_0 \in \mathbb{R}$ (see
3.5 in \cite{ChudnovskySeymour} or for a proof in the case of $\deg f = \deg g$, see \cite[Theorem 2']{Fell} or
\cite[Theorem~5.2]{Wag92}).
Secondly, it is immediate from \eqref{eq:roots-interlacing} that $f$ interlaces $g$
if and only if $0 \le n_g(x_0) - n_f(x_0) \le 1$ for
 all $x_0 \in \mathbb{R}$. In addition, we also have that
 $n_{xf}(x_0) = n_{f}(x_0) + \chi(x_0 \le 0)$. Since all roots of $f$ and $g$ are nonpositive, we may assume that $x_0 \le 0$.
 
Finally, it can easily be seen that the following two conditions are equivalent, which completes the proof. 

\begin{enumerate}[label={(\roman*)}]
\item $\left|n_f(x_0) - n_g(x_0)\right| \le 1$ and $\left|(n_f(x_0)+1) - n_g(x_0)\right| \le 1.$
\item $0 \le n_g(x_0) - n_f(x_0) \le 1.$
\end{enumerate}

\end{proof}


\begin{rem}
Some further results on the connection of interlacing and compatibility appeared recently in 
\cite{Liu12}.
\end{rem}

Lemma~\ref{lem:interlacing} together with Theorem~\ref{thm:compatible} implies
the following result of Haglund, Ono, and Wagner \cite[Lemma~8]{HOW99}.
\begin{cor} 
\label{cor:HOW}
Let $f_1, \dotsc, f_m \in \reals[x]$ be real-rooted polynomials with nonnegative
coefficients, and such that $f_i$ interlaces $f_j$ for all $1\le i < j \le m.$
Let $b_1, \dotsc, b_m \ge 0$ and $c_1, \dotsc, c_m \ge 0$ be such that 
$b_ic_{i+1} \leq c_ib_{i+1}$ for all $1 \le i \le m-1.$ Then 
$c_1f_1 + \dotsb + c_mf_m$ interlaces $b_1f_1 + \dotsb + b_mf_m.$
\end{cor}

In fact, since the $\s$-Eulerian polynomials have all positive coefficients this implies that
the $n$th polynomial interlaces the $(n+1)$th.
\begin{thm}
For any sequence $\s$ of positive integers and any positive integer $n$, we have that
\[ \poly^{(\s)}_n(x) \quad \textrm{interlaces}\quad \poly^{(\s)}_{n+1}(x). \]
\end{thm}
\begin{proof}
By definition, $\poly^{(\s)}_n(x)$ has only nonnegative coefficients, so we can apply 
Corollary~\ref{cor:HOW}. Set $m = s_{n+1}$, $c_1 = 1$, $c_2 = \dotso = c_{m} = 0$, 
$b_1 = \dotso = b_{m} = 1$ and $f_i = \poly^{(\s)}_{n+1,i-1}(x)$ for all $1\le i \le m$
to get that $\poly^{(\s)}_{n}(x) = \poly^{(\s)}_{n+1,0}(x) = f_1$ interlaces 
$f_1 + \dotsb + f_m = \sum_{i=0}^{m-1} \poly^{(\s)}_{n+1,i}(x) = \poly^{(\s)}_{n+1}(x).$

\end{proof}

\section{Applications}
\label{sec:applications}

In this section, we show that Theorem~\ref{thm:main} contains as special cases 
several existing real-rootedness results on (generalized) Eulerian polynomials, as well as some results which appear to be new. In particular, we prove the real-rootedness of the Eulerian
polynomials of type $D$ in Subsection~\ref{subsec:coxeter}.

\subsection{Permutations}

We first show that Theorem~\ref{thm:main} implies the real-rootedness of the familiar
Eulerian polynomials, a result known since Frobenius \cite{Fro}.

For $\pi \in \Sn$, let $\D \pi$ be the descent set of $\pi$,
\[
\D \pi = \{ i \in \{1, \ldots, n-1\} \mid \pi_i > \pi_{i+1}\},
\]
and let $\inv \pi$ be the number of {\em inversions} of $\pi$:
\[
\inv \pi \ =  \left |\{(i,j) \mid 1 \leq i < j \leq n \ {\rm and} \ \pi_i > \pi_j \} \right |.
\]
We will make use of the following bijection between $\Sn$ and
$\I_n^{(1,2, \ldots, n)}$
which was proved in \cite[Lemma 1]{SS} to have the properties claimed.
\begin{lem}
The mapping $\phi: \Sn \rightarrow \I_{n}^{(1,2, \ldots, n)}$ defined by
$\phi(\pi) = \boldsymbol{t} = (t_1, t_2, \ldots, t_n)$ for $\pi = (\pi_1, \dotsc, \pi_n)$ as
\[
t_i \ = \ \left|\{j \in  \{1,2, \ldots, i-1\} \mid   \pi_j>\pi_i\}\right|
\]
is a bijection satisfying both $\D \pi   =  \A \boldsymbol{t}$ and 
$\inv \pi   =  |\boldsymbol{t}| = t_1+t_2+ \cdots + t_n$.
\label{Desinv}
\end{lem}
\begin{cor}
For $n \geq 1$, the Eulerian polynomial,
\[
  A_n(x) =\sum_{\pi \in \Sn} x^{\des \pi}  ,
\]
has only real roots.
\end{cor}
\begin{proof}
By Lemma \ref{Desinv},
$A_n(x) 
= \E_n^{(1,2, \ldots, n)}(x),
$
which has all roots real by Theorem~\ref{thm:main}
with $\s=(1,2, \ldots, n)$.

\end{proof}

\subsection{Signed permutations}
\label{sec:signedpermutations}
Let $\Bn$ denote  the  hyperoctahedral group, whose elements
are the   signed permutations 
of $\{1,2, \ldots, n\}$.  Each $\sigma \in \Bn$ has the form
$(\pm \pi_1, \pm \pi_2, \ldots, \pm \pi_n)$ where $\pi = (\pi_1, \dotsc, \pi_n) \in \Sn$.

In defining the notion a of ``descent'' on $\Bn$, various orderings
 have been used in the literature.  
 In this subsection, we will assume the 
 \[-1 <_B -2 <_B \cdots <_B -n <_B 0 <_B 1 <_B 2 <_B \cdots <_B n\]
 ordering, since 
 it generalizes naturally to the wreath products discussed in the next subsection.
(For another ordering used for $\Bn$, see Subsection~\ref{subsec:coxeter}.)

Let $\sigma$ be a signed permutation of length $n$. 
An index $i \in \{0,1, \ldots, n-1\}$ is a {\em descent} of $\sigma$
if $\sigma_i >_B \sigma_{i+1}$, where $\sigma_0:=0$.
Let $\desB \sigma$ denote the number of
descents of $\sigma \in \Bn$.

There is 
a correspondence between statistics on signed permutations and statistics on 
inversion sequences $\I_n^{(\s)}$ with $\s = (2,4,6, \dotsc)$. The following was shown in \cite[eq.~(26)]{SS}.
\begin{lem}
\begin{equation}
\sum_{t \geq 0} (2t+1)^n x^t \ = \
\frac{\E_n^{(2,4,\ldots, 2n)}(x)}
{(1-x)^{n+1}}\,.
\label{signed_perms}
\end{equation}
\label{lem:signed_perms}
\end{lem}
On the other hand, the infinite series in (\ref{signed_perms}) was shown  
by Brenti in \cite[Theorem~3.4]{brenti94}
to satisfy:
\begin{equation}
\sum_{t \geq 0} (2t+1)^n x^t \ = \
\frac{\sum_{\sigma \in \Bn}x^{\desB(\sigma)}}
{(1-x)^{n+1}}\,.
\label{steingrimsson}
\end{equation}
So, we have the following result, originally due to
Brenti \cite[Corollary~3.7]{brenti94}.
\begin{cor}
\label{cor:typeB}
The descent polynomial for signed permutations,
\[
B_n(x)  \ := \  \sum_{\sigma \in \Bn}x^{\desB(\sigma)},
\]
has all real roots.
\end{cor}
\begin{proof}
Combining \eqref{signed_perms} and \eqref{steingrimsson},
$B_n(x) 
\ = \ \E_n^{(2,4,\ldots, 2n)}(x)$.
The result follows with $\s=(2,4,\ldots, 2n)$ from Theorem~\ref{thm:main}.

\end{proof}

\subsection{$k$-colored permutations}

For a positive integer $k$,
the wreath product $\w$,
of a cyclic group, $\integers_k$, of order $k$, and the symmetric group $\Sn$,
generalizes both $\Sn$ (the case $k=1$) and $\Bn$ ($k=2$).
We regard $\w$ 
as the set of $k$-colored permutations, or as pairs $(\pi, \xi)$,
written as
\[
\ps \ = \
(\pi_1^{\xi_1}, \pi_2^{\xi_2}, \ldots, \pi_n^{\xi_n}),
\]
 where $\pi=(\pi_1, \ldots, \pi_n) \in \Sn$ and
$\xi = (\xi_1, \ldots, \xi_n) \in \{0,1, \ldots, k-1\}^n$.

The {\em descent set} of $\ps \in \w$ is
\begin{align}
\D \ps  & \ = \ \{i \in \{0, \ldots, n-1\} \mid \xi_i < \xi_{i+1} \  {\rm or} \  \xi_i = \xi_{i+1} \ {\rm and} \ \pi_i > \pi_{i+1}\},
\label{des}
\end{align}
with the convention that $\pi_0=\xi_0=0$.
Let $\des \ps =\left|\D \ps\right|$ denote the number of descents.
Note that this definition of $\des$ agrees with $\des$ on $\Sn$ when $k=1$,
and with $\desB$ on $\Bn$ when $k=2$.
The descent polynomial for $G_{n,k} := \w$ is defined analogously as
\[
G_{n,k}(x) \ :=  \ \sum_{\ps \in \w} x^{\des \ps }.
\]

As we now describe, the statistics on $\w$ are related to
statistics on $\s$-inversion sequences, 
$\I_n^{(\s)}$, with $\s=(k,2k, \ldots, nk)$.  
The following bijection was proven in \cite[Theorem 3]{PS2} to
map the descent set on colored permutations, $\w$, 
to the ascent set on inversion sequences $\I_n^{(k,2k, \ldots, nk)}$.
\begin{lem}
 For each pair $(n,k)$ with $n \geq 1$, $k \geq 1$,
define
\[
\Theta: \ \ \w \ \longrightarrow \ \I_n^{(k,2k, \ldots, nk)}
\]
by
\begin{equation}
\e \ = \ \Theta(\pi_1^{\xi_1}, \pi_2^{\xi_2}, \ldots, \pi_n^{\xi_n})\\
\ =  \
(\xi_1 + t_1,\   2 \xi_2 + t_2,\    \ldots,\   n \xi_n + t_n),
\end{equation}
where $(t_1, t_2, \ldots, t_n) = \phi(\pi)$, for $\phi$ defined on $\Sn$ as in  Lemma \ref{Desinv}.

Then
\[
 \A \e  \  = \  \D \ps.
\]
\label{lem:theta}
\end{lem}
The following result is originally due to Steimgr\'{i}msson \cite[Theorem~3.19]{Stein}.
\begin{cor}
 For each pair $(n,k)$ with $n \geq 1$, $k \geq 1$, the descent polynomial
of $\w$ has all roots real.
\end{cor}
\begin{proof}
By Lemma \ref{lem:theta},
$G_{n,k}(x) \ = \ 
 \E_n^{(k,2k, \ldots, nk)}(x)$,
so the result follows from  Theorem~\ref{thm:main}
with $\s=(k,2k, \ldots, nk)$.

\end{proof}

In Section~\ref{sec:q-analogs}, we will use the fact that the bijection  $\Theta$
of Lemma \ref{lem:theta} relates other  statistics
of $\w$ and $\I_n^{(k,2k, \ldots, nk)}$,  to
show that several $q$-analogs of $G_{n,k}(x)$ are real-rooted for all positive $q$, settling some open questions.

\subsection{Finite Coxeter groups}
\label{subsec:coxeter}
The symmetric  group and the hyperoctahedral group are examples of finite Coxeter groups.
The descent statistic can be extended to all such groups (with the appropriate choice
of generators), and hence one can define
the Eulerian polynomials for all finite Coxeter groups, sometimes called Coxeter
systems (see \cite{BB05, brenti94}). 
Brenti showed that these polynomials have only real roots for type $B$ (see 
Corollary~\ref{cor:typeB}) and the exceptional 
groups and conjectured that this is the case in general \cite[Conjecture~5.2]{brenti94}. 
\begin{conj} 
\label{conj:Coxeter}
The Eulerian polynomials for all finite Coxeter groups have only real roots. 
\end{conj}
Brenti also showed by a simple argument that it is enough to check this for irreducible 
finite Coxeter groups. Combining this with the above results reduced 
Conjecture~\ref{conj:Coxeter} to the case of even-signed permutations \cite[Conjecture~5.1]{brenti94}.
\begin{conj} 
\label{conj:typeD}
The Eulerian polynomials of type $D$ have only real roots. 
\end{conj}

In this subsection, we give the first proof of Conjecture~\ref{conj:typeD}.
To be precise, we view
the Coxeter group of type $B$ (resp.~$D$) of rank $n$, denoted by $\Bn$ (resp.~$\Dn$),
as the set of signed (resp. even-signed) permutations of the set $\{1, \dotsc, n\}$.
The type $B$ and $D$ descents have the following simple combinatorial interpretation
(see \cite{brenti94, BB05}).
For a signed (resp. even-signed) permutation $\sigma$ given in its ``window notation" 
$(\sigma_1, \dotsc, \sigma_n)$, let
\begin{align}
\D_B \sigma &= \{ i \in \{1,\dotsc, n-1\} \mid \sigma_i > \sigma_{i+1}\} \cup
\{0 \mid \mathrm{if}\; \sigma_1 < 0\},
\label{def:B-des}
\\
\D_D \sigma &=  \{ i \in \{1,\dotsc, n-1\} \mid \sigma_i > \sigma_{i+1}\} \cup
\{0 \mid \mathrm{if}\; \sigma_1+\sigma_2 < 0\}.
\label{eq:Ddes}
\end{align}

Let us start with a simple observation.  Note that the  type $D$ 
descent statistic, $\D_D$, defined in (\ref{eq:Ddes}) can be extended 
to all signed permutations. Furthermore, $\D_D$ is equidistributed over
even-signed and odd-signed permutations. In other words, we have the following equality.
\begin{lem} For $n\ge 2,$
\[
\sum_{\sigma \in \Bn} x^{\des_D \sigma} = 2 \sum_{\sigma \in \Dn}x^{\des_D \sigma}.
\]
\label{lem:involution}
\end{lem}
\begin{proof}
The involution on $\Bn$ that swaps the values $1$ and $-1$ in (the window notation of) 
$\sigma\in\Bn$ is a bijection between $\Dn$ and $\Bn\setminus \Dn$ that preserves 
the type $D$ descent statistic whenever $n\ge 2$.

\end{proof}

Therefore, in order to avoid dealing with the parity of the signs and to allow for simpler 
recurrences, we will be working instead with the polynomial 
\begin{align} 
\label{Tdef}
T_n(x) & = \sum_{\sigma \in \Bn}x^{\des_D \sigma}. 
\end{align} 
Clearly, $T_n(x)$ has all roots real if and only if $D_n(x)$ does (even for $n=1$, in the trivial case not covered
by Lemma~\ref{lem:involution}, since $T_1(x) = x+1$ and $D_1(x)=1$). This observation 
allows us to focus our attention on signed permutations,
with the goal of showing $T_n(x)$ has all real roots.

We will prove the following inversion sequence representation of $T_n(x)$.
For $\e = (e_1, \dotsc, e_n) \in \I^{(2,4,\dots, 2n)}_n$, the type $D$ ascent set of
$\e$ is defined as
\begin{equation}
\A_D \e =  \left\{i\in \{1,\dotsc, n-1\} \Bigm| \frac{e_i}{i} < \frac{e_{i+1}}{i+1}\right\} \cup \left\{0 \Bigm| {\rm if} \  e_1 + \frac{e_2}{2}\ge \frac{3}{2}\right\}.
\label{eq:D-Asc}
\end{equation}
Let
\[\asc_D \e = \left|\A_D \e\right|.
\]

\begin{lem} 
\label{lem:invseqD}
For $n \ge 1,$
\[
T_n(x) = \sum_{\e \in \I^{(2,4,\dots, 2n)}_n} x^{\asc_D \e}\,.
\]
\label{lem:Tn-invseq}
\end{lem}

We will also make use of the following basic but practical observation.
\begin{lem}
\label{lem:observation}
Let $a,b,p$ be nonnegative integers such that $0 \le a/p < 1$ and  $0 \leq b/(p+1) < 1$.
Then
\[\frac{a}{p} < \frac{b}{p+1} \Longleftrightarrow a < b.\]
\end{lem}
\begin{proof}
If $a < b$, then $a+1 \leq b$.  Thus, since $a < p,$
$
(p+1)a = pa + a < pa + p = p(a+1) \leq p b.
$
So, $(p+1)a < pb$.
Conversely, if $a \geq b$, then
$(p+1)a \geq (p+1)b > p b,
$
so $(p+1)a > p b$.

\end{proof}

Clearly, the set of signed permutations, $\Bn$, has the same cardinality
as the set of ``type $B$'' inversion sequences, $\I^{(2,4,\dots, 2n)}_n$. 
Next, we define a bijection $\psi$ between these sets  
that maps type $D$ descents in signed permutations 
to type $D$ ascents  in the inversion sequences.  We will prove several other properties of 
$\psi$ as well.  Some will be used to establish the real-rootedness of $T_n(x)$---and
hence $D_n(x)$---others will be needed in Section~\ref{subsec:affine} for the affine Eulerian polynomials.

Throughout this subsection we will assume
the natural ordering of integers, \[-n < \dotsb <-1 < 0 < 1 < \dotsb < n.\]

For $\sigma=(\sigma_1, \ldots, \sigma_n) \in \Bn$, let
$(t_1, \dotsc, t_n)= \phi(|\sigma_1|, \dotsc, |\sigma_n|)$ where $\phi$ is the 
map defined in Lemma~\ref{Desinv} and $(|\sigma_1|, \dotsc, |\sigma_n|)$ denotes the underlying
permutation in $\Sn$. Define the map $\psi:  \Bn \rightarrow \I^{(2,4,\dots, 2n)}_n$ as follows. Let
$
\psi(\sigma) = (e_1, \ldots, e_n),
$
where, for all $1\le i\le n$, \[e_i=\begin{cases} t_i&$ if\, $\sigma_i > 0\,, \\ 2i-1-t_i&$ if\, $\sigma_i < 0\,.\end{cases}\]

\begin{thm}
\label{bijection}
The map $\psi: \Bn \rightarrow \I^{(2,4,\dots, 2n)}_n$ is a bijection satisfying the following properties.
\begin{enumerate}[label=\emph{(\arabic*)}, ref=(\arabic*)]
\item
$\sigma_1 < 0$ if and only if $e_1 > 0$. \label{cond:typeB}
\item
$\sigma_n > 0$ if and only if $e_n < n$. \label{cond:typeBtilde}
\item
$\sigma_1 + \sigma_2 < 0$ if and only if $e_1 + e_2/2 \geq 3/2$. \label{cond:typeD}
\item $\sigma_i > \sigma_{i+1}$ if and only if
$e_i/i < e_{i+1}/(i+1)$, for $1\le i\le n-1$. \label{cond:typeA}
\item
$\sigma_{n-1} + \sigma_n > 0$ if and only if $e_{n-1}/(n-1) + e_n/n <
 (2n-1)/n$. \label{cond:typeDtilde}
\end{enumerate}
\end{thm}
\begin{proof} 
Note that $\sigma_i < 0$ if and only if $e_i \geq i$
which proves \ref{cond:typeB} and \ref{cond:typeBtilde}.
Moreover, this shows that the map $\psi$ is a bijection since $\phi$ is.

\begin{enumerate}
\setcounter{enumi}{2}
\item It is not too hard to see that it is sufficient to verify this claim 
for all $\sigma \in \mathfrak{B}_2$. See Table~\ref{table}.
\begin{table}[ht]
\begin{center}
\begin{tabular}{|c||c|c|c|}

\hline
$\sigma\in\mathfrak{B}_2$ &$\e \in I_2^{(2,4)}$ & $\A_D \e$ & $\asc_D \e$ \\
\hline
(1,2)&(0,0) & $\{ \ \}$ & 0 \\
\hline
(-1,2)&(1,0) & $\{ \ \}$ & 0 \\
\hline
(2,1) &(0,1) & $\{ 1 \}$ & 1 \\
\hline
(-2,1)&(1,1) & $\{ 0 \}$ & 1 \\
\hline
(2,-1)&(0,2) & $\{ 1 \}$ & 1 \\
\hline
(-2,-1)&(1,2) & $\{ 0 \}$ & 1 \\
\hline
(1,-2)&(0,3) & $\{ 0,1 \}$ & 2 \\
\hline
(-1,-2)&(1,3) & $\{ 0,1 \}$ & 2 \\
\hline
\end{tabular}
\end{center}
\caption{An example of the bijection for $n=2$.}
\label{table}
\end{table}

\item
To prove this claim, we consider four cases, based the signs of $\sigma_i$ and $\sigma_{i+1}$.

\begin{enumerate}
\item{If $\sigma_i > 0$ and $\sigma_{i+1} > 0$},
then $e_i = t_i < i$ and $e_{i+1}=t_{i+1} < i+1$.
By Lemma~\ref{Desinv}, $\sigma_i > \sigma_{i+1}$ if and only if $t_i < t_{i+1}$, 
i.e, if and only if $e_i < e_{i+1}$.
By Lemma~\ref{lem:observation}, this is equivalent to 
$e_i/i < e_{i+1}/(i+1)$.

\item{If $\sigma_i < 0$ and $\sigma_{i+1} < 0$},
then $e_i =2i-1- t_i $ and $e_{i+1}=2(i+1)-1-t_{i+1} $.
Now $\sigma_i > \sigma_{i+1}$ if and only if $|\sigma_i| < |\sigma_{i+1}|$,
which,
applying Lemma~\ref{Desinv}, is equivalent to $t_i \geq t_{i+1}$.

If $t_i \geq t_{i+1}$,
\[
\frac{e_i}{i} = 2 - \frac{t_i+1}{i} \leq 2 - \frac{t_{i+1}+1}{i} 
< 2 - \frac{t_{i+1}+1}{i+1} = \frac{e_{i+1}}{i+1}.
\]
On the other hand, if $t_i < t_{i+1}$, then $t_i+1 \leq  t_{i+1}$ and 
by Lemma~\ref{lem:observation}, $t_{i+1}/i < (t_{i+1}+1)/(i+1)$, so
\[
\frac{e_i}{i} = 2 - \frac{t_i+1}{i} \geq 2 - \frac{t_{i+1}}{i}
> 2 - \frac{t_{i+1}+1}{i+1} =
\frac{e_{i+1}}{i+1}.
\]

\item{If $\sigma_i < 0<\sigma_{i+1}$}, then
$e_i =2i-1- t_i $ and $e_{i+1}=t_{i+1} \leq  i$.
Since $t_i \leq i-1$, $e_i \geq 2i-1-(i-1)=i.$
Thus we have
\[
\frac{e_i}{i} \geq 1> \frac{i}{i+1}
\geq \frac{e_{i+1}}{i+1}.
\]

\item{If $\sigma_i > 0>\sigma_{i+1}$}, then
$e_i = t_i < i $ and $e_{i+1}= 2(i+1)-1-t_{i+1}.$
Since $t_{i+1} \leq i$, $e_{i+1} \geq 2(i+1)-1-(i)=i+1.$
Thus we have
\[
\frac{e_i}{i} < 1 
\leq \frac{e_{i+1}}{i+1}.
\]
\end{enumerate}
\item 
 Since $t_n = n-|\sigma_n|$,
\[
e_n = \left \{
\begin{array}{ll}
n-\sigma_n & {\rm if} \ \sigma_n > 0\\
n-1 + |\sigma_n| & {\rm if} \ \sigma_n < 0.\\
\end{array}
\right.
\]
Note that $t_{n-1} = n-|\sigma_{n-1}| - \chi(|\sigma_n| > |\sigma_{n-1}|)$.
Thus,
\[
e_{n-1} = \left \{
\begin{array}{ll}
n-\sigma_{n-1} - \chi(|\sigma_n| > |\sigma_{n-1}|)  & {\rm if} \ \sigma_{n-1} > 0\\
n-3 + |\sigma_{n-1}| + \chi(|\sigma_n| > |\sigma_{n-1}|) & {\rm if} \ \sigma_{n-1} < 0.\\
\end{array}
\right.
\]

First assume $\sigma_{n-1}+\sigma_{n} > 0$.
Then either
(i) $\sigma_{n-1} > 0$ and $\sigma_{n} > 0$, and so
\[
\frac{e_{n-1}}{n-1}+ \frac{e_{n}}{n}
\leq \frac{n-2}{n-1}+ \frac{n-1}{n} < \frac{2n-1}{n}\,,
\]

or
(ii) $\sigma_{n-1} > 0$ and $\sigma_{n} < 0$ with $1 \leq  |\sigma_{n}| < \sigma_{n-1} \leq n,$
in which case

\begin{align*}
\frac{e_{n-1}}{n-1}+ \frac{e_{n}}{n}
&=
\frac{n-\sigma_{n-1}}{n-1} + \frac{n-1+|\sigma_{n}|}{n}\\
&=
\frac{2n-1}{n} + \left(\frac{|\sigma_{n}|}{n} - \frac{\sigma_{n-1}-1}{n-1} \right)
<
\frac{2n-1}{n}\,,
\end{align*}

or
(iii) $\sigma_{n-1} < 0$ and $\sigma_{n} > 0$ with 
 $ 1\le |\sigma_{n-1}| < \sigma_{n}\le n,$ so that

\begin{align*}
\frac{e_{n-1}}{n-1}+ \frac{e_{n}}{n}
&=
\frac{n-2+|\sigma_{n-1}|}{n-1} + \frac{n-\sigma_{n}}{n}\\
&=
\frac{2n-1}{n} + \left(\frac{|\sigma_{n-1}|-1}{n-1} - \frac{\sigma_n-1}{n}\right)
<
\frac{2n-1}{n}\,.
\end{align*}
The last inequality in both cases (ii) and (iii) follows by Lemma~\ref{lem:observation}.

Now assume $\sigma_{n-1}+\sigma_{n} < 0$.
Then either
(iv) $\sigma_{n-1} < 0$ and $\sigma_{n} < 0$, so
\[ \frac{e_{n-1}}{n-1}+ \frac{e_{n}}{n} \ge 1 + 1 > \frac{2n-1}{n}\,,\]

or
(v) $\sigma_{n-1} < 0$ and $\sigma_{n} > 0$ with $1\le |\sigma_{n}| < \sigma_{n-1} \le n$, then 
\begin{align*}
\frac{e_{n-1}}{n-1}+ \frac{e_{n}}{n}
&=\frac{n-3 +|\sigma_{n-1}| }{n-1} + \frac{n-|\sigma_{n}|}{n}\\
&= \frac{2n-1}{n} + \left (\frac{|\sigma_{n-1}| - 2}{n-1} - \frac{|\sigma_{n}| - 1}{n} \right ) \ge \frac{2n-1}{n}\,,
\end{align*}

 or
(vi) $\sigma_{n-1} > 0$ and $\sigma_{n} < 0$ with
 $ 1\le |\sigma_{n-1}| < \sigma_{n} \le n$, then
\begin{align*}
\frac{e_{n-1}}{n-1}+ \frac{e_{n}}{n} &= 
\frac{n-1 -|\sigma_{n-1}| }{n-1} + \frac{n+|\sigma_{n}|-1}{n}
\\&= \frac{2n-1}{n} + \left (\frac{|\sigma_{n}| }{n} - \frac{|\sigma_{n-1}| }{n-1} \right )
 \ge \frac{2n-1}{n},
\end{align*}
Again, we applied Lemma~\ref{lem:observation} in the last steps of (v) and (vi).
\end{enumerate}

\end{proof}

\begin{proof}[Proof of Lemma~\ref{lem:invseqD}]
Follows from parts \ref{cond:typeD} 
and \ref{cond:typeA} 
of Theorem~\ref{bijection}.

\end{proof}

\begin{cor} Parts \ref{cond:typeB} and \ref{cond:typeA} of Theorem~\ref{bijection} can be used to give an alternative
proof of Corollary~\ref{cor:typeB}.
\end{cor}

\begin{lem} For $n \geq 2$ and $0 \leq i < 2(n+1),$
\[T_{n+1,i}(x) = \sum_{\ell=0}^{\lceil ni/(n+1)\rceil - 1} xT_{n,\ell}(x) 
+ \sum_{\ell=\lceil ni/(n+1)\rceil}^{2n-1}T_{n,\ell}(x),\]
with initial conditions
$T_{2,0}(x)=2$,
$T_{2,1}(x)= T_{2,2}(x)=2x$, and
$T_{2,3}(x)=2x^2$.
\label{lem:T-recurrence}
\end{lem}
\begin{proof}
The initial conditions can be checked from the Table~\ref{table}. 
Now suppose $n \geq 2$ and  $\e=(e_1, \ldots, e_{n+1}) \in \I_{n+1}^{(2,4,\dotsc,2n+2)}$ with $e_{n+1}=i$.  
Then, by the definition of the type $D$ ascent set,
$n \in \A_D \e$ if and only if
$e_n/n < i/(n+1)$ or, equivalently, whenever
$0 \leq 
\ell \leq  \lceil ni/(n+1)\rceil - 1$.
So,
\[
\asc_D \e = \asc_D(e_1, \ldots, e_n) + \chi(e_n \leq \lceil ni/(n+1)\rceil - 1). 
\]
We conclude the proof by letting $\ell = e_n$. 

\end{proof}

Finally, we are in position to prove Brenti's conjecture (Conjecture~\ref{conj:typeD}). 
\begin{thm} For $n\ge 2$, the polynomial
$T_n(x)$ has only real roots.
In fact, for $0 \leq i < 2n$, $T_{n,i}(x)$ has only real roots.
\label{thm:typeD}
\end{thm}

\begin{proof}
We prove the statement by induction for $n\ge 4$ using Theorem~\ref{thm:compatible}. For $n=2$ and $n=3$ the hypotheses
of Theorem~\ref{thm:compatible} do not hold. We need to check the cases $n=2$ and $n=3$ separately.

Clearly, $T_2(x) = 2(x+1)^2$ has only real roots, but the polynomials 
$T_{2,0}(x)=2,T_{2,1}(x)=T_{2,2}(x)=2x,T_{2,3}(x)=2x^2$ 
fail to be compatible, since $T_{2,0}(x) + T_{2,3}(x)$ has no real roots.

Using the recurrence given in Lemma~\ref{lem:T-recurrence} we can easily compute 
$T_{n,i}(x)$ for $n=3$. While $T_3(x) = 2(x^3+11x^2+11x+1)$ has only real roots, 
the polynomials $T_{3,0}(x) = 2(x+1)^2$, $T_{3,1}(x) = 2x(x+3)$, 
$T_{3,2}(x) =  T_{3,3}(x) = 4x(x+1)$,
$T_{3,4}(x)  = 2x(3x+1)$,
$T_{3,5}(x) = 2x(x+1)^2$ are not compatible, e.g., $T_{3,0}(x)+T_{3,4}(x)$ has no real roots. 

However, iterating one more time, we obtain the following eight polynomials (the approximate values of their roots
are also given for the reader's convenience):
\[
\begin{tabular}{rlll}
$T_{4,0}(x) = $& $2(x+1)(x^2+10x+1)$ && $\{-9.899, -1, -0.101\}$\\
$T_{4,1}(x) = $& $4x(x+1)(x+5)$&& $\{-5,-1,0\}$ \\
$T_{4,2}(x) = $& $2x(3x^2+14x+7)$&&$\{-4.097, -0.569, 0\}$\\
$T_{4,3}(x) = $& $2x(5x^2+14x+5)$&&$\{-2.380,-0.420, 0\}$ \\
$T_{4,4}(x) = $& $2x(5x^2+14x+5)$&&$\{-2.380,-0.420, 0\}$ \\
$T_{4,5}(x) = $& $2x(7x^2+14x+3)$&&$\{-1.756, -0.244, 0\}$\\
$T_{4,6}(x) = $& $4x(x+1)(5x+1)$&&$\{-1,-0.2, 0\}$\\
$T_{4,7}(x) = $& $2x(x+1)(x^2+10x+1)$&&$\{-9.899, -1, -0.101,0\}$.
\end{tabular}\]

We need to show that these eight polynomials
are indeed pairwise compatible and also that $xT_{4,i}(x)$ and $T_{4,j}(x)$ are
compatible for all $0\le i < j \le 7$.
By Lemma~\ref{lem:interlacing}, this can be done by checking the roots explicitly to verify that 
$T_{4,i}(x)$ interlaces $T_{4,j}(x)$
for all $0 \le i < j \le 7$. 
Proceeding by induction on $n$, successive applications of Theorem~\ref{thm:compatible} gives us that for all $n\ge 4$ the 
polynomials
$T_{n,0}(x), \dotsc, T_{n, 2n-1}(x)$ are pairwise compatible and also that $xT_{n,i}(x)$ and $T_{n,j}(x)$ are compatible 
for all $0\le i < j \le 2n-1$. 
In particular, the former is equivalent to saying that 
these $2n$ polynomials are compatible. Therefore, their sum, $T_n(x)$, has only real roots for all $n\ge 4$ as well.

\end{proof}

\subsection{Affine descents in Weyl groups}
\label{subsec:affine}
Recently, Dilks, Petersen and Stembridge defined and studied Eulerian-like polynomials associated
to irreducible affine Weyl groups. In \cite{DPS09}, they define these ``affine'' Eulerian polynomials 
as generating functions for ``affine descents'' over the corresponding finite Weyl group. An affine descent is similar to an ordinary descent in a Weyl group, except that the reflection corresponding 
to the highest root (in the underlying root system) may also contribute a descent, depending on its effect on length. 

Dilks, Petersen and Stembridge observed that these polynomials have interesting properties similar
to their counterparts for the Coxeter groups and proposed a companion conjecture to Brenti's 
conjecture. 
\begin{conj}[Conjecture~4.1 in \cite{DPS09}]
The affine Eulerian polynomials for all finite Weyl groups have only real roots.
\label{conj:weyl}
\end{conj}
As they pointed out the type $A$ and $C$ affine Eulerian polynomials 
were already known to be multiples of the classical Eulerian polynomial and hence,
have only real roots. 
In this section, we prove one of the remaining cases, for type $B$ 
(the type $D$ case remains open). 

The affine Eulerian polynomial of type $B$ is defined in \cite[Section 5.3]{DPS09} as 
the generating function of the ``affine descents'' over the corresponding finite Weyl group, $\Bn$,
\[ \widetilde{B}_n(x) = \sum_{\sigma\in \Bn} x^{\widetilde{\des}_B \sigma}, \]
where for a signed permutation $\sigma = (\sigma_1, \dotsc,  \sigma_{n}) \in \Bn$ the 
affine descent statistic is computed as
\[\widetilde{\des}_B \sigma = \chi(\sigma_1 <0) +  \left | \{1\le i \le n-1 \mid \sigma_i > \sigma_{i+1}\}\right | + \chi(\sigma_{n-1} + \sigma_{n} > 0).\]

Notice the affine Eulerian polynomial of type 
$B$ is intimately related to the type $D$ Eulerian polynomial in the following way.
\begin{thm} For $n \ge 2$,
\[\widetilde{B}_n(x) = T_{n+1,n+1}(x)\,,\]
where $T_{n,i}(x)$ is the refined Eulerian polynomial of type $D$ defined in (\ref{Tdef}).
\end{thm}
\begin{proof}
It is easy to see under the involution $(\sigma_1, \ldots, \sigma_n) \mapsto
(-\sigma_n, \ldots, -\sigma_1)$, that $\widetilde{\des}_B$ has the same distribution over
$\Bn$ as the statistic
\[ \widetilde{\stat}_B \sigma = \chi(\sigma_n > 0) +  \left | \{1\le i \le n-1 \mid \sigma_i > \sigma_{i+1}\}\right | + \chi(\sigma_{2} + \sigma_{1} < 0).\]
From Theorem~\ref{bijection} part \ref{cond:typeD} it follows that 
$\sigma_{2} + \sigma_{1} < 0$ is equivalent to $e_1 + e_2/2 > 3/2$ and from part \ref{cond:typeBtilde} we have that 
$\sigma_{n} > 0$ if and only if 
$e_n < n$.
Note $e_n < n$ is equivalent to $e_n/n < 1 = (n+1)/(n+1)$.
So, $\widetilde{B}_n(x) = T_{n+1,n+1}(x)$.

\end{proof}
\begin{cor} For $n\ge 2$,
$\widetilde{B}_n(x)$ has only real roots.
\end{cor}
\begin{proof}
Follows from the fact that $T_{n,i}(x)$ have only real roots (see Theorem~\ref{thm:typeD}).

\end{proof}

As we mentioned earlier  there is an analogous conjecture for type $D$ which
remains unsolved.
\begin{conj}[\cite{DPS09}] For $\sigma \in \Dn,$ let \[\widetilde{\des}_D \sigma = \chi(\sigma_1+\sigma_2 <0) +  \left | \{1\le i \le n-1 \mid \sigma_i > \sigma_{i+1}\}\right | + \chi(\sigma_{n-1} + \sigma_{n} > 0)\,.\] 
Then the affine Eulerian polynomial of type $D$
\[ \sum_{\sigma \in \Dn} x^{\widetilde{\des}_D}\]
has only real roots. 
\end{conj}
By Theorem \ref{bijection} (parts \ref{cond:typeD}, 
\ref{cond:typeA} and \ref{cond:typeDtilde}) we can at least express the type 
$D$ affine Eulerian polynomial in terms of ascent statistics on inversion sequences.
\begin{cor} The type $D$ affine Eulerian polynomial satisfies
\[
2 \sum_{\sigma \in \Dn} x^{\widetilde{\des}_D \sigma} = 
\sum_{\e \in {\I_n^{(2,4,\dotsc,2n)}}} x^{\widetilde{\asc}_D e},
\]
where the type $D$ affine ascent statistic for $\e \in{\I_n^{(2,4,\dots,2n)}}$ is given by
\begin{align*}
\widetilde{\asc}_D\, \e =& \;\chi(e_1 + e_2/2 \ge 3/2) + \left |  \left\{1\le i\le n-1 \;\Big|\; \frac{e_i}{i} < \frac{e_{i+1}}{i+1}\right\} \right |  \\ &+
\chi( e_{n-1}/(n-1) + e_n/n < (2n-1)/n).
\end{align*}
\end{cor}

\subsection{$k$-ary words}
The {\em $k$-ary words of length $n$} are the elements of the set
$\{0,1, \ldots, k-1\}^n$. Define an ascent statistic for
$w \in \{0,1, \ldots, k-1\}^n $ by
\[
\asc w = |\left\{i \in \{0, 1, \ldots, n-1\} \mid w_i < w_{i+1}\right\}| ,
\]   
with the convention that $w_0=0$.
\begin{cor}
The ascent polynomial for $k$-ary words,
\[
\sum_{w \in  \{0,1, \ldots, k-1\}^n} x^{\asc w},
\] has all real roots.
\end{cor}
\begin{proof}
Clearly, using the identity mapping from $\{0,1, \ldots, k-1\}^n $ to $ \I_{n}^{(k,k, \ldots, k)}$,
$\sum_{w \in \{0,1, \ldots, k-1\}^n } x^{\asc w} \ = \ 
E_n^{(k,k,\dots, k)}(x)\,.$
So, the result follows by setting $\s=(k,k, \ldots, k)$ in  Theorem~\ref{thm:main}.

\end{proof}
It was shown in \cite[Corollary 8]{SS} using Ehrhart theory that
\[
\frac{\sum_{\e \in \I_{n}^{(k,k, \ldots, k)}} x^{\asc \e}}{(1-x)^{n+1}} = 
\sum_{t \geq 0} \binom{n+kt}{n} x^t. 
\]

\begin{rem}
It was pointed out in \cite[Section 5]{DiaconisFulman} that
the above series arises as the Hilbert series of the $k$th Veronese
embedding of the coordinate ring of the full projective space 
$\mathbb{C}[x_1, \dotsc, x_{n+1}]$.
\end{rem}

\subsection{Excedances and number of cycles in permutations}
For a permutation $\pi \in \Sn$, the {\em excedance} number of $\pi$,
$\exc(\pi)$, is  defined by
\[
\exc(\pi) = |\left\{i \in \{1,2, \ldots, n\} \mid \pi(i) > i\right\}|,\]
 and the {\em cycle number} of $\pi$, $\cyc(\pi)$, is  the number of cycles in the disjoint cycle representation of
$\pi$.
Let
\[
A^{\exc,\cyc}_n(x,y) =  \sum_{\pi \in \Sn} x^{\exc \pi} y^{\cyc \pi}.
\]
It was proven by Brenti in \cite[Theorem 7.5]{brenti2000} that $A^{\exc,\cyc}_n(x,y)$ has all roots real for every positive $y \in \reals$.  This  was extended by 
Br\"and\'en in \cite[Theorem 6.3]{MR2218995} to include values of $y$ for which $n+y \le 0$.

In \cite{gopal},
$A^{\exc,\cyc}_n(x,1/k)$ was shown to be related to inversion sequences.  This will
allow us to deduce the real-rootedness in this special case (when $y = 1/k$) from  Theorem~\ref{thm:main}.
\begin{cor}
For every positive integer, $k$, the polynomial
\[
A^{\exc,\cyc}_n(x,1/k) =  \sum_{\pi \in \Sn} x^{\exc \pi} k^{-\cyc \pi}
\]
has only real roots.
\end{cor}
\begin{proof}
Let  $\s=(k+1,2k+1, \ldots,(n-1)k+1)$.
It was shown in \cite[Theorems 3 and 6]{gopal},  that for every positive integer $k$,
\[ \sum_{\e \in \I_{n}^{(\s)}} x^{\asc \e} = k^n A^{\exc,\cyc}_n(x,1/k).\]
The corollary follows from Theorem~\ref{thm:main} with $\s=(k+1,2k+1, \ldots,(n-1)k+1)$.

\end{proof}

\begin{rem}
It is well-known that the pair of statistics $(\exc, \cyc)$ is equidistributed with $(\des, \mathrm{lrm})$,
where $\mathrm{lrm}$ is the number of left-to-right minima in a permutation.
\end{rem}

\subsection{Multiset permutations}
Simion \cite[Section 2]{MR728500} showed that
for any $n$-element multiset, $M$, the descent polynomial for  the set of permutations,
$P(M)$, of $M$ has only real roots.  A {\em descent} in a multiset permutation,
 $\pi \in P(M)$ is an index $i \in \{1,2, \ldots, n-1\}$ such that
$\pi_i > \pi_{i+1}$.

When $M=\{1,1,2,2, \ldots, n,n\}$, there is a connection with
inversion sequences.  Let $\s$ be the sequence  $\s=(1,1,3,2,5,3,7,4, \ldots)$,
where for $i \geq 1$, $s_{2i}=i$ and $s_{2i-1}=2i-1$.
Observe that the number of  $\s$-inversion sequences of length $2n$ is the same as the number of permutations of $\{1,1,2,2, \ldots, n,n\}$:
\[
\left |\I_{2n}^{(1,1,3,2,5,3,7,4, \ldots, 2n-1,n)} \right | \ = \ \frac{(2n)!}{2^n} =  \ \left |P(\{1,1,2,2, \ldots, n,n\})\right |.
\]
We discovered
that the distribution of ascents on the first set is equal to the distribution of descents on the second set.
\begin{thm}
\[
\sum_{\pi \in P(\{1,1,2,2, \ldots, n,n\})} x^{\des \pi}
= \sum_{\e \in \I_{2n}^{(1,1,3,2,5,3, \ldots, 2n-1, n)}} x^{\asc \e}.
\]
\label{multiset_perms}
\end{thm}
\begin{proof}
It was shown in 
\cite[Theorem 14]{SS} that
\[
\sum_{t \geq 0}\left ( \frac{(t+1)(t+2)}{2} \right )^{ n} x^t \ = \
\frac{\sum_{\e \in \I_{2n}^{(1,1,3,2,5,3, \ldots, 2n-1,n)}} x^{\asc \e}}
{(1-x)^{2n+1}}.
\]
MacMahon \cite[Volume 2, Chapter IV, p. 211, \S462]{MR2417935} showed that
\[
\frac{\sum_{\pi \in P(\{1^{p_1}, \dotsc, n^{p_n}\})} x^{\des \pi}}{(1-x)^{1+\sum_i p_i}} =
\sum_{t\ge 0}\frac{(t+1)\dotsb (t+p_1) \dotsb (t+1)\dotsb(t+p_n)}{p_1! \cdot \dotsb \cdot p_n!} \,x^t\,.
\]
In particular, 
when $p_i = 2$ for all $i$, this implies
\[
\frac{\sum_{\pi \in P(\{1,1,2,2, \ldots, n,n\})} x^{\des \pi}}{(1-x)^{2n+1}} \ = \   \sum_{t \geq 0}\left ( \frac{(t+1)(t+2)}{2} \right )^{ n} x^t. 
\]

\end{proof}
We thus obtain the following special case of Simion's result as a corollary of Theorem~\ref{thm:main}.
\begin{cor}
The polynomial
\[
\sum_{\pi \in P(\{1,1,2,2, \ldots, n,n\})} x^{\des \pi}
\]
has only real roots.
\end{cor}
The sequence $\s=(1,1,3,2,5,3,7,4, \ldots)$ was  studied in \cite{CSS}, where it
was shown  that the
$\s$-lecture hall partitions lead to a new finite model for
the {\em Little G\"ollnitz identities}. 
There was a companion sequence, $\s=(1,4,3,8,5,12, \ldots, 2n-1,4n)$ defined by
$s_{2i}=4i$, $s_{2i+1}=2i+1$, which we now consider in the context of multiset permutations.

Let $P^{\pm}(\{1,1,2,2, \ldots, n,n\})$ be the set of all \emph{signed} permutations of the multiset
$\{1,1,2,2, \ldots, n,n\}$.  The elements are those of the form
$(\pm \pi_1, \pm \pi_2, \ldots, \pm \pi_{2n})$, where
$( \pi_1,  \pi_2, \ldots,  \pi_{2n}) \in P(\{1,1,2,2, \ldots, n,n\})$.
Note that
\[
\left | P^{\pm}(\{1,1,2,2, \ldots, n,n\}) \right | \ = \ \frac{(2n)!}{2^n} 2^{2n} \ = \ 2^n(2n)! \  = \ 
\left | \I_{2n}^{(1,4,3,8, \ldots, 2n-1,4n)} \right | .
\]
From our experiments it appears that
distribution of descents on the first set is equal to the distribution of ascents on the second set, and we make that conjecture.
\begin{conj}
\[
\sum_{\pi \in P^{\pm}(\{1,1,2,2, \ldots, n,n\})} x^{\des \pi}
= \sum_{\e \in \I_{2n}^{(1,4,3,8, \ldots, 2n-1, 4n)}} x^{\asc \e}.
\]
\end{conj}
If this  conjecture is true, it would follow as a corollary of Theorem~\ref{thm:main}
 that the descent polynomial for the signed permutations of the multiset
$\{1,1,2,2, \ldots, n,n\}$ has all real roots.

It was shown in
\cite[Theorem 13]{SS} that
\[
\sum_{t \geq 0}\left ( (t+1)(2t+1)\right  )^{ n} x^t \ = \
\frac{\sum_{\e \in \I_{2n}^{(1,4,3,8,5,12, \ldots, 2n-1,4n)}} x^{\asc \e}}
{(1-x)^{2n+1}}.
\]
It may be possible to show that $\sum_{\pi \in P^{\pm}(\{1,1,2,2, \ldots, n,n\})} x^{\des \pi}$
satisfies the same identity. Finally, it would be interesting to investigate $q$-analogs of the above
identities and possibly the conjecture.

\section{Geometric consequences}
\label{sec:geometry}
In this section, we describe some geometric consequences of Theorem~\ref{thm:main}.

\subsection{The $h$-polynomials of finite Coxeter complexes and reduced Steinberg tori}
There is a natural simplicial complex associated with a finite Coxeter group $W$ and 
its reflection representation.  The Coxeter complex of $W$ is the simplicial complex 
$\Sigma = \Sigma(W)$ formed as the intersection of a unit sphere and the reflecting 
hyperplanes of $W$.

The $f$-polynomial of a $(d-1)$-dimensional simplicial complex
$\Delta$ is the generating function for the dimensions
of the faces of the complex:
\[f(\Delta,x) = \sum_{F \in \Delta} x^{\dim F + 1}. \]

The $h$-polynomial of $\Delta$ is a transformation of the
$f$-polynomial: 
\[h(\Delta,x) = (1-x)^df\left(\Delta,\frac{x}{1-x}\right). \]

In fact, the $h$-polynomial of a Coxeter complex $\Sigma(W)$
is the Eulerian polynomial of type $W$ discussed in 
Subsection~\ref{subsec:coxeter}:
\[h(\Sigma(W),x) = \sum_{\sigma\in W} x^{\des_W(\sigma)}. \]

In \cite{DPS09}, Dilks, Petersen, and Stembridge defined
a Boolean cell complex, called the \emph{reduced Steinberg
torus} whose $h$-polynomial is the affine Eulerian polynomial of type $W$ 
discussed in Subsection~\ref{subsec:affine}.
 
A curious property that implies unimodality and is implied by
real-rootedness (under certain conditions) is called \emph{$\gamma$-nonnegativity}. Every polynomial $h(x)$ of degree $n$ 
that is palindromic, i.e., satisfies $h(x) = x^nh(1/x)$, can be written uniquely in the form
\[h(x) = \sum_{i=0}^{\lfloor n/2\rfloor} \gamma_i x^i(1+x)^{n-2i}.\]
If the coefficients $\gamma_i$ are nonnegative for $0\le i \le n/2$, then we say that $h(x)$
is $\gamma$-nonnegative. For a (palindromic) polynomial with only positive coefficients real-rootedness implies
$\gamma$-nonnegativity (see \cite[Lemma~4.1]{Br04}, \cite[Remark~3.1.1]{Gal05}).

Since, all the Eulerian and the affine Eulerian polynomials of type $W$ are palindromic with positive
coefficients, the real-rootedness of all but one of these polynomials (recall that the case of the affine 
$D$-Eulerian polynomial which is still open) implies their $\gamma$-nonnegativity. 
The $\gamma$-nonnegativity was established for various types combinatorially: for type $B$ 
in \cite[4.15]{Petersen}, for type $D$ in 
\cite[Theorem~6.9]{Chow},  (see also \cite[Theorem~1.2]{Stembridge}); and for the affine Eulerian
polynomials in \cite[Theorem~4.2]{DPS09}.

\subsection{The $\h^*$-polynomials of $\s$-lecture hall polytopes}

For background, the {\em Ehrhart series} of a polytope $\PP$ in $\reals^n$ is the series
\[
\sum_{t \geq 0}  |t\PP \cap \integers^n| x^t,
\]
where
$t\PP$ is the $t$-fold {\em dilation} of $\PP$:
\[
t\PP = \left\{(t\la_1, t \la_2, \ldots, t \la_n) \mid (\la_1, \la_2, \ldots, \la_n) \in \PP\right\}.
\]
So, $i(\PP,t) := |t\PP \cap \integers^n|$ is the number of points in $t\PP$, all of whose coordinates
are integer.

For the rest of this discussion assume that the polytope $\PP$ is integral, that is, all of its vertices have integer coordinates.  Then $i(\PP,t)$ 
is a {\em polynomial} in $t$
and the Ehrhart series of $\PP$ has the form
\[
\sum_{t \geq 0} i(\PP,t) x^t = \frac{\h(x)}{(1-x)^{n+1}},
\]
for a polynomial $\h(x) = h_0 + h_1x + \cdots h_dx^d$,  where $h_d \not = 0$ and $d \leq n$.  The polynomial $\h(x)$ is known as the {\em $\h^*$-polynomial} of $\PP$ \cite{ehrhart1,ehrhart2}.  

By Stanley's {\em Nonnegativity Theorem} 
 \cite[Theorem~2.1]{nonnegthm} 
 the coefficients of its $\h^*$-polynomial are
nonnegative.
The sequence of coefficients $h_1,h_2, \ldots, h_d$ of $\h(x)$
is called the {\em $\h^*$-vector} of $\PP$. (Alternative names also appear in the literature, such as Ehrhart $h$-vector, $\delta$-vector.)

\begin{example} Consider the triangle in the plane with vertices $(0,0)$, $(1,2)$, $(2,1)$, formally, let
\begin{equation}
\PP = \{(\la_1, \la_2) \in \reals^2 \mid \la_1 \leq 2 \la_2, \ \la_2 \leq 2 \la_1, \ {\rm and} \ \la_1 + \la_2 \leq 3\}.
\label{polytope_ex}
\end{equation}
Then
\[
i(\PP,t) \  = | t\PP \cap \integers^2 | \ =  \  1 + 3/2 t + 3/2 t^2,
\]
and the Ehrhart series of $\PP$ is
\begin{equation}
\sum_{t \geq 0}(1 + 3/2 t + 3/2 t^2) \, x^t \  = \  \frac{x^2+x+1}{(1-x)^3}.
\label{polytope_es}
\end{equation}

So, the $\h^*$-polynomial of the  polytope $\PP$ is $\h(x)=x^2+x+1$  and its $\h^*$-vector is
$[1,1,1]$, which is nonnegative, symmetric, and unimodal.
\end{example}

The $\h^*$-vector of a convex polytope with integer vertices need not be
symmetric or unimodal.
Although there has been much progress in the direction of
characterizing those
polytopes whose $\h^*$-vector is unimodal (see, e.g., \cite{Athanasiadis,BrunsRomer,MustataPayne,Stanley1980}) this is still
an open question.

However, we can  use
Theorem~\ref{thm:main} to answer the question for the following class of
polytopes associated with $\s$-inversion sequences.

The {\em $\s$-lecture hall polytope} $\PP_n^{(\s)}$ is defined by
\[
\PP_n^{(\s)} = \left\{ (\la_1, \la_2, \dotsc, \la_n) \in \reals^n \ \Big| \
0 \leq \frac{\la_{1}}{s_{1}}
\leq \frac{\la_{2}}{s_{2}} \leq \cdots
\leq \frac{\la_{n}}{s_{n}} \leq 1 \right\},
\]
where $\s$ is an arbitrary sequence of positive integers.

The following is a special case of Theorem 5 in \cite{SS}.
\begin{lem}  For any sequence $\s$ of positive integers,
\[
\sum_{t \geq 0} i(\PP_n^{(\s)},t) x^t =
\frac{\E_n^{(\s)}(x)}{(1-x)^{n+1}}.
\]
\label{SSlem}
\end{lem}

So combining Lemma~\ref{SSlem} with Theorem~\ref{thm:main} we have:
\begin{cor} For any sequence $\s$ of positive integers,
the $\h^*$-polynomial of the $\s$-lecture hall polytope has all roots real.
\end{cor}
The $\s$-lecture hall polytopes are special in this regard, even among lattice simplices.
The polytope $\PP$ of the example (\ref{polytope_ex})  is a simplex in $\reals^2$ with
integer vertices, but its $\h^*$-polynomial, $x^2 + x + 1$, does not have real roots.

The sequence of coefficients of a real polynomial with only real roots is log-concave, and---if the coefficients are nonnegative---it is also unimodal. This is an easy corollary of a classic result, often referred to as Newton's inequality. 
Thus, we have the following.
\begin{cor} For any sequence $\s$ of positive integers,
The $\h^*$-vector of the  $\s$-lecture hall polytope
is unimodal and log-concave.
\end{cor}

\begin{rem}
The $\h^*$-vector an  $\s$-lecture hall polytope need not be {\em symmetric}. For example,
\[
\E_n^{(1,3,5)}(x) \ = \ 1 + 10x + 4x^2.
\]
\end{rem}

\section{$(p,q)$-analogs of $\s$-Eulerian polynomials}
\label{sec:q-analogs}
In this section, we define $(p,q)$-analogs of the $\s$-Eulerian polynomials and
show that they have real rools for every positive $p,q\in \reals$. 

In addition to the statistics 
$\A \e$,  $\asc \e$ and $|\e| = \sum_i e_i$ on 
$\s$-inversion sequences, 
we define  a new statistic, related to the major index on permutations.
For $\e \in \I_{n}^{(\s)}$, let
\begin{align*}
\amaj \e = & \sum_{j \in \A \e}(n-j).
\end{align*}
 For a sequence of positive integers $\s$, and a positive integer $n$, 
define a $(p,q)$-analog of the $\s$-Eulerian polynomials as
\begin{equation}
\E^{(\s)}_{n}(x,p,q) = \sum_{\e \in \I_{n}^{(\s)}}
x^{\asc \e}q^{\amaj \e} p^{|\e|}\,,
\label{eq:pq-analog}
\end{equation}
and for $0 \leq i < s_{n}$, define its refinement as before
\[
 \poly^{(\s)}_{n,i}(x,p,q) =  
 \sum_{ \e \in \I_{n}^{(\s)}} \chi(e_n = i)\,
x^{\asc \e}q^{\amaj \e} p^{|\e|}.
\]
\begin{lem}
For $n \geq 1$ and $0 \leq i < s_{n+1}$,
\[
\poly^{(\s)}_{n+1,i}(x,p,q) = p^i \left ( \sum_{j=0}^{\ell-1}
 xq \poly^{(\s)}_{n,j}(xq,p,q)  +
\sum_{j=\ell}^{s_{n}-1}
\poly^{(\s)}_{n,j}(xq,p,q) \right )\,
\]
where $\ell = \lceil i s_{n}/s_{n+1} \rceil$, and
with initial conditions $\poly^{(\s)}_{1,0}(x,p,q)=1$ and $\poly^{(\s)}_{1,i}(x,p,q)=xqp^i$ for $i > 0$.

\label{Pqzrec}
\end{lem}
\begin{proof}
For $\e=(e_1, \dotsc, e_{n}, i) \in  \I_{n+1}^{(\s)}$
we have that $n \in \A \e$ if and only if
$e_{n}/s_{n} < i/s_{n+1}$, that is, if
$0 \leq e_{n} \leq \ell-1$. Thus, the statistics change accordingly:
\begin{align*}
|(e_1, \dotsc, e_{n}, i)| &= |(e_1, \dotsc, e_{n})| + i\,, \\
\asc (e_1, \dotsc, e_{n}, i) &=\asc (e_1, \dotsc, e_{n}) + \chi(e_{n} \leq \ell-1)\,,\\ 
\amaj(e_1, \dotsc, e_{n}, i) &= \amaj(e_1, \dotsc, e_{n}) + \asc(e_1, \dotsc, e_{n}) + \chi(e_{n} \leq \ell-1)\,.
\end{align*}
For the initial conditions,
$0 \in \A \e$ if and only if $e_1 > 0$, in which case
$\asc \e = 1$,  $\amaj \e = n-0=1$ and $|\e|=e_1 = i$.

\end{proof}

Since $
\E_n^{(\s)}(x,p,q) = \sum_{i=0}^{s_n-1} \poly^{(\s)}_{n,i}(x,p,q),$
we have the following result.
\begin{thm}
Let $\s = (s_1, s_2, \dotsc)$ be a sequence of positive integers. For any positive integer $n$ and positive real 
numbers $p$ and $q$, the polynomial $\E^{(\s)}_{n}(x,p,q)$, defined in \eqref{eq:pq-analog}, has only real roots.
\label{thm:pq-analog}
\end{thm}

\begin{proof} By Lemma~\ref{Pqzrec}, the statement is a special case of the following theorem with
$b_0 = 1$, $b_i = qp^i$ for $0 < i< s_1$, $c_{n,i} = q$, and $d_{n,i} = p^i$ for all $n\ge 1$ and $
0\le i < s_n$.

\end{proof}


\begin{thm}
Given a sequence of positive integers, $\s = \{s_i\}_{i=1}^\infty$, and positive numbers $\{b_i \mid 0\le i < s_1\}$, $\{c_{n,i} \mid n\ge 1,\, 0\le i < s_n\}$ and $\{d_{n,i} \mid n\ge 1,\, 0\le i < s_n\}$ define the polynomials $R^{(\s)}_{n,i}(x)$ as follows.
Let $R^{(\s)}_{1,0}(x) = b_0$, $R^{(\s)}_{1,i}(x) = b_ix$ for $1\le i < s_1$, 
and 
for $n \ge 1$ and $0 \leq i < s_n$, let 
\[
R^{(\s)}_{n+1,i}(x) = d_{n,i} \left ( 
\sum_{j=0}^{\ell-1} c_{n,j}x R^{(\s)}_{n,j}(c_{n,j}x) + \sum_{j=\ell}^{s_{n}-1}  R^{(\s)}_{n,j}(c_{n,j}x) \right ), 
\]
with $\ell = \lceil i s_{n+1}/s_{n} \rceil$.

Then for all $n\ge 1$, $\sum_{i} R^{(\s)}_{n,i}(x)$ is real-rooted.
\label{thm:qcompatible}
\end{thm}

\begin{proof} We apply the method of Section~\ref{sec:main}. 
The polynomials 
$\{R^{(\s)}_{n,i}(x)\}_{ 0 \leq i < s_n }$ are compatible by
Theorem~\ref{thm:compatible}. Note that the transformation $f(x) \mapsto
c f(qx)$ preserves real-rootedness and the sign of the leading coefficient
for $c,q > 0.$

\end{proof}

\subsection{The MacMahon--Carlitz $q$-analog for $\Sn$}

We apply Theorem \ref{thm:pq-analog} now to the case of permutations. In particular, we
show for the first time that the MacMahon--Carlitz $q$-Eulerian polynomial
is real-rooted for $q>0$, a result conjectured by Chow and Gessel in \cite{ChowGessel}.

For $\pi \in \Sn$, let
$\maj \pi = \sum_{j \in \D \pi} j$, and $\comaj \pi = \sum_{j \in \D \pi} (n-j).$
It was conjectured in \cite{ChowGessel}, that for any $q > 0$, the polynomial 
\[
A^{\maj}_n(x,q) = \sum_{\pi \in \Sn} x^{\des \pi}q^{\maj \pi},
\]
has all roots real.

We now  settle this conjecture as part of the  following theorem.
\begin{thm}
All of the following $q$-analogs of the Eulerian polynomials 
$A_n(x)= \sum_{\pi \in \Sn} x^{\des \pi}$,  are real-rooted whenever $q>0$:
\begin{align}
A^{\inv}_n(x,q)=& \sum_{\pi \in \Sn} x^{\des \pi} q^{\inv \pi}, \label{qE1}\\
A^{\comaj}_n(x,q)=& \sum_{\pi \in \Sn} x^{\des \pi} q^{\comaj \pi}, \label{qE2}\\
A^{\maj}_n(x,q)=& \sum_{\pi \in \Sn} x^{\des \pi} q^{\maj \pi} \label{qE3}.
\end{align}
The last of these is known as the MacMahon--Carlitz $q$-Eulerian polynomial.
\end{thm}
\begin{proof}
The bijection 
$\phi: \Sn \rightarrow \I_{n}^{(1,2, \ldots, n)}$ of Lemma \ref{Desinv}
maps the pair of statistics $(\D, \inv)$ on permutations to the pair 
$(\A, | \ |)$ on inversion sequences,
so we have
\[
\sum_{\e \in \I_{n}^{(1,2, \ldots, n)}} x^{\asc \e} q^{\amaj \e} p^{|\e|} \ 
  =\ 
\sum_{\pi \in \Sn} x^{\des \pi} q^{\comaj \pi} p^{\inv \pi}.
\]
In particular, \[\E_n^{(1,2, \ldots, n)}(x,p,1) = A^{\inv}_n(x,p) \quad 
\mathrm{and}\quad
\E_n^{(1,2, \ldots, n)}(x,1,q) = A^{\comaj}_n(x,q)\]
and by Theorem~\ref{thm:pq-analog} they are real-rooted for $p >0$ and $q > 0$,
respectively.  As for
$A^{\maj}_n(x,q)$,
observe that the mapping $\Sn \rightarrow \Sn$ defined by
\[
\pi = (\pi_1, \pi_2, \ldots, \pi_n) \  \mapsto \ 
\pi' = (n+1-\pi_n, n+1- \pi_{n-1}, \ldots, n+1-\pi_1) \]
satisfies $\D \pi' = \{n-j \mid j \in \D \pi \}.$ 
Thus 
\[
\sum_{\pi \in \Sn} x^{\des \pi} q^{\maj \pi} \ = \
\sum_{\pi' \in \Sn} x^{\des \pi'} q^{\comaj \pi'} \ = \
\sum_{\pi \in \Sn} x^{\des \pi} q^{\comaj \pi},
\]
so $(\des,\maj)$ and $(\des, \comaj)$ have the same joint distribution on $\Sn$, i.e.,
$A^{\maj}_n(x,q) = A^{\comaj}_n(x,q)$.

\end{proof}

\begin{cor}
For $q > 0 $ the coefficients of the polynomials $A^{\maj}_n(x,q)$ defined in (\ref{qE3}) form a
unimodal and log-concave sequence.
\end{cor}
Partial results on the unimodality of these coefficients were obtained in \cite[Proposition 9]{HJZ2013}.

\subsection{Some $q$-analogs for $\Bn$ and $\w$}

We can also apply Theorem~\ref{thm:pq-analog} to signed permutations and
wreath products, making use of Lemma~\ref{lem:theta}.
Extend $\maj$ and $\comaj$ to $\w$ by defining them for $\ps \in \w$ as
$\maj \ps = \sum_{j \in \D \ps} j$ and $\comaj \ps = \sum_{j \in \D \ps} (n-j).$

The statistic {\em flag inversion number}  is defined
for $\ps \in \w$ by
\[
\finv \ps \ = \inv \pi + \sum_{i=1}^n i \xi_i.
\]

Recall the bijection $\Theta$ of Lemma~\ref{lem:theta} that maps the descent set on $\w$ to
the ascent set on $\I_n^{(k,2k, \ldots, nk)}$. The following additional property of $\Theta$ 
was shown in \cite[Theorem~3]{PS2}.
\begin{lem}
The bijection $\Theta: \w \ \longrightarrow \ \I_n^{(k,2k, \ldots, nk)}$ of Lemma~\ref{lem:theta} satisfies
\[
\finv \ps \ =  \left | \Theta(\ps) \right |.
\]
\label{thetafinv}
\end{lem}
\begin{thm}
The following $q$-analogs of the descent polynomial $G_{n,k}(x)$ have all roots real for $q > 0$:
\begin{align}
G^{\finv}_{n,k}(x,q) \ = &  \ \sum_{\ps \in \w} x^{\des \ps } q^{\finv \ps}, \label{Wq1}\\
G^{\comaj}_{n,k}(x,q) \ = &  \ \sum_{\ps \in \w} x^{\des \ps } q^{\comaj \ps},\label{Wq2}\\
G^{\maj}_{n,k}(x,q) \ = &  \ \sum_{\ps \in \w} x^{\des \ps } q^{\maj \ps}. \label{Wq3}
\end{align}
\end{thm}
\begin{proof}
By Lemma \ref{thetafinv}, 
\[\E_n^{(k,2k, \ldots, nk)}(x,1,p) = G^{\finv}_{n,k}(x,p)\quad
\mathrm{and} \quad 
\E_n^{(k,2k, \ldots, nk)}(x,q,1) = G^{\comaj}_{n,k}(x,q),\] 
both of which are real-rooted by
Theorem~\ref{thm:pq-analog}.
In contrast to the case for $\Sn$, 
when $k>1$, the polynomials (\ref{Wq2}) and (\ref{Wq3}) are not the same.  However,
\[\sum_{\ps \in \w} x^{\des \ps } q^{\maj \ps} \ = \ 
\sum_{\ps \in \w} (xq^n)^{\des \ps } q^{-\comaj \ps}\,.\]
So, $G^{\maj}_{n,k}(x,q) = \E_n^{(k,2k, \ldots, nk)}(xq^n,1,1/q)$ and, by Theorem \ref{thm:pq-analog},
has all roots real.

\end{proof}

\begin{cor}
For $q > 0$ the coefficients of the polynomials $G^{\maj}_n(x,q)$ defined in (\ref{Wq3}) form a
unimodal and log-concave sequence.
\end{cor}
Partial results on the unimodality of these coefficients were obtained in \cite[Proposition 10]{HJZ2013}.

\subsection{Euler--Mahonian $q$-analogs for $\Bn$ and $\w$}
\label{subsec:euler-mahonian}

In the case of $\Bn$ and $\w$, a different  $q$-analog of the
Eulerian polynomial is based on the {\em flag major index} statistic
\cite{AdinRoichman},
$\fmaj$, which is defined 
for $\ps \in \w$ by
\[
\fmaj \ps \ = k \, \comaj \ps -  \sum_{i=1}^n  \xi_i.
\]
This definition differs a bit from those appearing elsewhere because of the appearance of $\comaj$, but it was shown
in \cite{PS2} to be equivalent, e.g., to the definition in  \cite{ChowGessel,ChowMansour}.

In contrast to $\comaj$ and $\maj$ from the previous section,
$\fmaj$ is {\em Mahonian}, i.e., it has the same distribution as
``length'' on $\w$ \cite{AdinRoichman}.  For that reason, the polynomials
\[
G^{\fmaj}_{n,k}(x,q)= \sum_{\ps \in \w} x^{\des \ps}q^{\fmaj \ps}
\]
are referred to as the {\em Euler--Mahonian polynomials} for $\w$.
It was conjectured by Chow and Mansour in
\cite{ChowMansour} that $G^{\fmaj}_{n,k}(x,q)$
has all roots real for $q > 0$.
In the special case, $k=2$, of signed permutations, $\Bn$, the conjecture was
made by Chow and Gessel in 
 \cite{ChowGessel} that for any $q > 0$, the
polynomial $B_n(x,q)$, defined by
\[
B_n(x,q) =  \sum_{\sigma \in \Bn} x^{\des \sigma}q^{\fmaj \sigma},
\]
has all real roots.

We can now settle these conjectures. 
\begin{thm}
For the wreath product groups, $\w$, the Euler--Mahonian polynomials
\[
G^{\fmaj}_{n,k}(x,q)= \sum_{\ps \in \w} x^{\des \ps}q^{\fmaj \ps}
\]
have all real roots for $q > 0$.
\label{wreathreal}
\end{thm}

To prove this theorem, 
we first relate  $\fmaj$ to a  statistic,
$\Ifmaj$, on the inversion sequences 
$\I_n^{(k,2k, \ldots, kn)}$.
For  $\e = (e_1, \dotsc, e_n) \in \I_n^{(k,2k, \ldots, nk)}$, define
\begin{align*}
\Ifmaj \e & \ = \  k \, \amaj \e - \sum_{j=1}^n \floor{\frac{e_j}{j}}.
\end{align*}

It was shown in \cite [Theorem 3]{PS2}
that the
bijection $\Theta: \w \rightarrow \I_n^{(k,2k, \ldots, nk)}$ of Lemma \ref{lem:theta},
which maps the descent set on $\w$ to
the ascent set on $\I_n^{(k,2k, \ldots, nk)}$ also satisfies
\[
\fmaj \ps = \Ifmaj \Theta(\ps).
\]
We thus have
\begin{equation}
G^{\fmaj}_{n,k}(x,q) 
\  = \ 
 \sum_{\e \in \I_n^{(k,2k, \ldots, nk)}} x^{\asc \e}q^{\Ifmaj \e}.  
\label{EMF}
\end{equation}
 For $n,k \geq 1$ and $0 \leq i < nk$, define
\[
G^{\fmaj}_{n,k,i}(x,q) =  \sum_{\e \in \I_n^{(k,2k, \ldots, nk)}}
\chi(e_n = i)\, x^{\asc \e}q^{\Ifmaj \e}.
\]
\begin{lem}
For $n,k \geq 1$ and $0 \leq i < (n+1)k$, 
\[
G^{\fmaj}_{n+1,k,i}(x,q) = q^{-\lfloor i/(n+1) \rfloor} \left ( \sum_{j=0}^{\ell-1}
 xq^k G^{\fmaj}_{n,k,j}(xq^k,q)  +
\sum_{\ell}^{s_{n}-1}
G^{\fmaj}_{n,k,j}(xq^k,q) \right )\,
\]
where $\ell = \lceil i n/(n+1) \rceil$, and
with initial conditions $G^{\fmaj}_{1,k,0}(x,q)=1$ and $G^{\fmaj}_{1,k,i}(x,q)=xq^{k-i}$ for $i > 0$.
\label{Frec}
\end{lem}
\begin{proof}
For $\e=(e_1, \dotsc, e_{n},i) \in  \I_{n+1}^{(k,2k, \dotsc, (n+1)k)}$
we have that
$n \in \A \e$ if and only if
$e_n/n < i/(n+1)$, that is
$0 \leq e_n \leq \ell-1$.
Similarly to the proof of Lemma~\ref{Pqzrec}, we see that by appending the value $i$ in the $(n+1)$th position
the statistics change as follows:
\begin{align*}
\asc (e_1, \dotsc, e_{n},i) &= \asc (e_1, \dotsc, e_{n}) + \chi(e_n \leq \ell-1)\,,\\
\Ifmaj (e_1, \dotsc, e_{n},i) & = \Ifmaj (e_1, \dotsc, e_{n}) + k \asc (e_1, \dotsc, e_{n},i) - \lfloor i/(n+1) \rfloor\,.
\end{align*}
For the initial conditions,
$0 \in \A \e$ if and only if $e_1 > 0$, in which case $\des \e = 1$ and $\Ifmaj \e = k \, \amaj \e - \lfloor i/1 \rfloor = k-i$, and
both are zero otherwise.

\end{proof}

\begin{proof}[Proof of Theorem~\ref{wreathreal}]

For fixed positive integer $k$ and real number $q$, setting $s_n = nk$ for all $n\ge 1$, $b_0=1$, $b_i = q^{k-i}$ for $0\le i < k$, $c_{n,j} = q^k$ for $0 \le j < nk$ and $d_{n,i}=q^{-\lfloor i/(n+1) \rfloor}$ for $0 \le i < nk$, 
in Theorem~\ref{thm:qcompatible} gives the
 recurrence of Lemma~\ref{Frec}.
Thus, by Theorem~\ref{thm:qcompatible}, 
\[
G^{\fmaj}_{n,k}(x,q)   
 \ = \
\sum_{i=0}^{nk-1} G^{\fmaj}_{n,k,i}(x,q)
\]
has all real roots.

\end{proof}  

\section*{Acknowledgments}
We thank the National Science Foundation (grant \#1202691 supporting the Triangle 
Lectures in Combinatorics) and the Simons Foundation (grant \#244963) for travel funding 
that faciltated the research presented here. The second author was also supported by the 
Knut and Alice Wallenberg Foundation.

Thanks to Thomas Pensyl for his contributions to Theorem~\ref{bijection}.
Special thanks to Christian Krattenthaler for his comments on a partial result presented
at the S\'eminaire Lotharingien de Combinatoire at Strobl that encouraged us
to develop our methods for the type $D$ case. We also thank Petter Br\"and\'en for helpful suggestions.   

We are grateful to the referee who did a careful reading of the paper and offered many corrections and suggestions to improve the presentation.

\bibliographystyle{plain}
\bibliography{roots}

\end{document}